\input ./style/arxiv-vmsta.cfg
\documentclass[numbers,compress]{vmsta}

\usepackage{eufrak,mathrsfs}

\volume{2}
\pubyear{2015}
\firstpage{29}
\lastpage{49}
\doi{10.15559/15-VMSTA24}

\startlocaldefs
\newcommand{\rrVert}{\Vert}
\newcommand{\llVert}{\Vert}

\endlocaldefs

\newtheorem{thm}{Theorem}
\newtheorem{lemma}{Lemma}
\newtheorem{cor}{Corollary}
\newtheorem{prop}{Proposition}
\newtheorem{Assumption}{Assumption}

\theoremstyle{definition}

\newtheorem{remark}{Remark}

\theoremstyle{remark}

\def\R{\mathbb{R}}
\def\E{\mathbb{E}}
\def\P{\mathbb{P}}
\def\F{\mathbb{F}}
\def\HH{ \EuFrak H}
\newcommand{\ud}{\mathrm{d}}

\allowdisplaybreaks

\begin{document}
\begin{frontmatter}

\title{Asymptotic normality of randomized periodogram for~estimating quadratic
variation in mixed Brownian--fractional Brownian model}

\author[a]{\inits{E.}\fnm{Ehsan}\snm{Azmoodeh}\corref{cor1}}\email{ehsan.azmoodeh@uni.lu}
\cortext[cor1]{Corresponding author.}

\author[b]{\inits{T.}\fnm{Tommi}\snm{Sottinen}}\email{tommi.sottinen@iki.fi}

\author[c]{\inits{L.}\fnm{Lauri}\snm{Viitasaari}}\email{lauri.viitasaari@aalto.fi}
\address[a]{Mathematics Research Unit, Luxembourg University,\\P.O. Box L-1359, Luxembourg}
\address[b]{Department of Mathematics and Statistics, University of Vaasa,\\P.O. Box 700, FIN-65101 Vaasa, Finland}
\address[c]{Department of Mathematics and System Analysis,\\Aalto University School of Science, Helsinki,\\
P.O. Box 11100, FIN-00076 Aalto, Finland\\
Department of Mathematics, Saarland University,\\Post-fach 151150, D-66041 Saarbr\"ucken, Germany}

\markboth{E.~Azmoodeh et al.}{Asymptotic normality of randomized
periodogram for estimating quadratic variation}

\begin{abstract}
We study asymptotic normality of the randomized periodogram estimator
of qua\-dratic variation in the mixed Brownian--fractional Brownian
model. In the semimartingale case, that is, where the Hurst parameter
$H$ of the fractional part satisfies $H\in(3/4,1)$, the central limit
theorem holds. In the nonsemimartingale case, that is, where
$H\in(1/2,3/4]$, the convergence toward the normal distribution with
a~nonzero mean still holds if $H=3/4$, whereas for the other values,
that~is, $H \in(1/2,3/4)$, the central convergence does not take
place. We also provide Berry--Esseen estimates for the estimator.
\end{abstract}

\begin{keyword}
Central limit theorem\sep
multiple Wiener integrals\sep
Malliavin calculus\sep
fractional Brownian motion\sep
quadratic variation\sep
randomized periodogram

\MSC[2010] 60G15, 60H07, 62F12\vspace*{-9pt}
\end{keyword}

\received{17 November 2014}
%
\revised{30 March 2015}
\accepted{24 April 2015}
\publishedonline{11 May 2015}
\end{frontmatter}

\section{Introduction and motivation}

The quadratic variation, or the pathwise volatility, of stochastic
processes is of\break paramount importance in mathematical finance. Indeed,
it was the major discovery of the celebrated article by Black and
Scholes \cite{bs73} that the prices of financial derivatives depend
only on the volatility of the underlying asset. In the Black--Scholes
model of geometric Brownian motion, the volatility simply means the
variance.
Later the Brownian model was extended to more general semimartingale
models. Delbaen and Schachermayer \cite{ds94,ds98} gave the final
word on the pricing of financial derivatives with semimartingales. In
all these models, the volatility simply meant the variance or the
semimartingale quadratic variance.
Now, due to the important article by F\"ollmer \cite{f81}, it is clear
that the variance is not the volatility. Instead, one should consider
the pathwise quadratic variation. This revelation and its implications
to mathematical finance has been studied, for example, in
\cite{bsv08,sk99}.

An important class of pricing models is the mixed Brownian--fractional
Brownian model. This is a model where the quadratic variation is
determined by the Brownian part and the correlation structure is
determined by the fractional Brownian part. Thus, this is a pricing
model that captures the long-range dependence while leaving the
Black--Scholes pricing formulas intact. The mixed Brownian--fractional
Brownian model has been studied in the pricing context, for example,
in \cite{am06,bsv07,bsv11}.

By the hedging paradigm the prices and hedges of financial derivative
depend only on the pathwise quadratic variation of the underlying
process. Consequently, the statistical estimation of the quadratic
variation is an important problem. One way to estimate the quadratic
variation is to use directly its definition by the so-called
\textit{realized quadratic variation}. Although the consistency result
(see \xch{Section}{Subsection}~\ref{rqv-1}) does not depend on a specific choice of
the sampling scheme, the asymptotic distribution does. There are
numerous articles that study the asymptotic behavior of realized
quadratic variation; see \cite{bn-sh-1, bn-s, jacod, fuka, hjy11} and
references therein. Another approach, suggested by Dzhaparidze and
Spreij \cite{d-s}, is to use the randomized periodogram estimator. In
\cite{d-s}, the case of semimartingales was studied. In
\cite{azm-val}, the randomized periodogram estimator was studied for
the mixed Brownian--fractional Brownian model, and the weak consistency
of the estimator was proved. This article investigates the asymptotic
normality of the randomized periodogram estimator for the mixed
Brownian--fractional Brownian model.

The rest of the paper is organized as follows. In Section
\ref{estimation}, we briefly introduce the two estimators for the
quadratic variation already mentioned.
In Section \ref{analysis}, we introduce the stochastic analysis for
Gaussian processes needed for our results. In particular, we introduce
the F\"ollmer pathwise calculus and Malliavin calculus. Section
\ref{beef} contains our main results: the central limit theorem for the
randomized periodogram estimator and an associated Berry--Esseen bound.
Finally, some technical calculations are deferred into Appendix~\ref{A}
and Appendix \ref{B}.

\section{Two methods for estimating quadratic variation}\label{estimation}

\subsection{Using discrete observations: realized quadratic
variation}\label{rqv-1}

It is well known that (see \cite[Chapter 6]{protter}) for a
semimartingale $X$, the bracket $[X,X]$ can be identified
with
\eject

\[
[X,X]_t = \P\hbox{-} \lim_{|\pi| \to0} \sum
_{t_k\in\pi} ( X_{t_k} - X _{t_{k-1}} ) ^2
,
\]
where $\pi=  \lbrace t_k : 0= t_0 < t_1 <\cdots< t_n =
t \rbrace$ is a partition of the interval $[0,t]$, $|\pi| = \max
 \lbrace t_k - t_{k-1} : t_k \in\pi \rbrace$, and $\P
\hbox{-} \lim$ means convergence in probability. Statistically
speaking, the sums of squared increments $($\textit{realized quadratic
variation}$)$ is a consistent estimator for the bracket as the volume
of observations tends to infinity. Barndorff-Nielsen and Shephard
\cite{bn-s} studied precision of the realized quadratic variation
estimator for a special class of continuous semimartingales. They
showed that sometimes the realized quadratic variation estimator can be
a rather noisy estimator. So one should seek for new estimators of the
quadratic variation.

\subsection{Using continuous observations: randomized periodogram}\label
{spectral-charac}

Dzhaparidze and Spreij \cite{d-s} suggested another characterization of
the bracket $[X,X]$. Let $\F^X$ be the filtration of $X$, and $\tau$
be a finite stopping time. For $\lambda\in\R$, define the
\textit{periodogram} $I_\tau(X;\lambda)$ of $X$ at $\tau$ by
\begin{align}
I_\tau(X;\lambda) :& = \bigg| \int_0^\tau e^{i\lambda s} \ud X_s \bigg|^2\nonumber\\
&= 2 \ \textbf{Re} \int_{0}^{\tau} \int_{0}^{t} e^{i \lambda(t-s)} \ud X_s \ud X_t + [X,X]_\tau\quad(\text{by It\^o formula}).\label{eq:perio}
\end{align}
Let $\xi$ be a symmetric random variable independent of the filtration
$\F^X$ with density $g_\xi$ and real characteristic function
$\varphi_\xi$. For given $L>0$, define the randomized periodogram by
\begin{equation}
\label{eq:p-random} \E_\xi I_\tau ( X; L\xi ) = \int
_\R I_\tau ( X; Lx ) g_\xi(x) \ud x .
\end{equation}
If the characteristic function $\varphi_\xi$ is of bounded variation,
then Dzhaparidze and Spreij have shown that we have the following
characterization of the bracket as $L \to\infty$:
\begin{equation}
\label{eq:d-s} \E_\xi I_\tau ( X; L\xi ) \stackrel{\P} {\to}
[X,X]_\tau .
\end{equation}

Recently, the convergence (\ref{eq:d-s}) is extended in \cite{azm-val}
to some class of stochastic processes which contains nonsemimartingales
in general. Let $W= \{W_{t}\}_{t\in[0,T]}$ be a standard Brownian
motion, and $B^{H}= \{B^{H}_{t}\}_{t \in[0,T]}$ be a fractional
Brownian motion with Hurst parameter $H \in(\frac{1}{2},1)$,
independent of the Brownian motion $W$. Define the mixed
Brownian--fractional Brownian motion $X_{t}$ by
\begin{equation*}
X_{t}= W_{t}+B^{H}_{t}, \quad t \in[0,T] .
\end{equation*}

\begin{remark}

It is known that $($see \cite{ch}$)$ the process $X$ is an $(\F
^X , \P)$-semi\-martingale if $H \in( \frac{3}4,1)$, and for $H\in
 ( \frac{1}2 , \frac{3}4  ]$, $X$ is not a semimartingale with
respect to its own filtration $\F^X$. Moreover, in both cases, we have
\begin{equation}
\label{eq:mixed-bracket} [X,X]_{t}= t .
\end{equation}
If the partitions in (\ref{eq:mixed-bracket}) are nested, that is, for
each $n$, we have $\pi^{(n)} \subset\pi^{(n+1)}$, then the convergence
can be strengthened to almost sure convergence. Hereafter, we always
assume that the sequences of partitions are nested.
\end{remark}

Given $\lambda\in\R$, define the periodogram of $X$ at $T$ as
$($\ref{eq:perio}$)$, that is,
\begin{equation*}
\begin{split}
I_{T}(X;\lambda) & = \bigg| \int
_{0}^{T}e^{i\lambda t} \ud X_{t} \bigg|
^{2}
\\
& = \bigg| e^{i\lambda T}X_{T}-i\lambda\int_{0}^{T}
X_{t}e^{i\lambda t} \ud t \bigg|^2
\\
& = X_T ^{2} + X_T \int_{0}^{T}
i\lambda\bigl( e^{i \lambda(T-t)} - e^{- i
\lambda(T-t)} \bigr) X_t \ud t
+ \lambda^2 \bigg| \int_{0}^{T}
e^{i \lambda
t} X_t \ud t \bigg| ^2. \end{split} %
\end{equation*}

Let $(\tilde{\varOmega},\tilde{\mathcal{F}},\tilde{\mathbb{P}})$ be
another probability space. We identify the $\sigma$-algebra
$\mathcal{F}$ with $\mathcal{F} \otimes\{ \phi, \tilde{\varOmega} \}$ on
the product space $(\varOmega\times\tilde{\varOmega}, \mathcal{F} \otimes
\tilde{\mathcal{F}},\mathbb{P}\otimes\tilde{\mathbb{P}})$. Let
$\xi:\tilde{\varOmega} \rightarrow\R$ be a~real symmetric random variable
with density $g_\xi$ and independent of the filtration $\F^X$.
For any positive real number $L$, define the randomized periodogram $\E
_{\xi} I_{T}(X;L\xi)$ as in (\ref{eq:p-random}) by
\begin{equation}
\label{eq:r-p-m-b-fb} \E_{\xi} I_{T}(X;L\xi):=\int
_{\mathbb{R}}I_{T}(X;Lx)g_\xi(x) \ud x,
\end{equation}
where the term $I_T(X;Lx)$ is understood as before. Azmoodeh and
Valkeila \cite{azm-val} proved the following:
\begin{thm}\label{t:main}
Assume that $X$ is a mixed Brownian--fractional Brownian motion,\break
$\E_\xi I_T(X; L\xi) $ be the randomized periodogram given by
\emph{(\ref{eq:r-p-m-b-fb})}, and
\[
\E\xi^2 < \infty.
\]
Then, as $L \to\infty$, we have
\begin{equation*}
\label{eq:main} \E_{\xi} I_{T}(X;L\xi) \stackrel{\P} {
\longrightarrow} [X,X]_T .
\end{equation*}
\end{thm}

\section{Stochastic analysis for Gaussian processes}\label{analysis}

\subsection{Pathwise It\^o formula}\label{path-follmer}

F{\"o}llmer \cite{f81} obtained a pathwise calculus for continuous
functions with finite\break quadratic variation. The next theorem
essentially belongs to F{\"o}llmer. For a~nice exposition and its use
in finance, see Sondermann \cite{sondermann}.

\begin{thm}[\cite{sondermann}] \label{t:f-ito}
Let $X :[0,T] \rightarrow\R$ be a continuous process with continuous
quadratic variation $[X,X]_{t}$, and let $F\in C^{2}(\R)$. Then for any
$t\in[0,T]$, the limit of the Riemann--Stieltjes sums
\begin{equation*}
\lim_{ |\pi| \to0} \sum_{
\begin{subarray}{1}
t_{i}\leqq t
\end{subarray}
}
F_x(X_{t_{i-1}}) (X_{t_{i}}-X_{t_{i-1}}):= \int
_{0}^{t}F_x(X_{s}) \ud
X_{s}
\end{equation*}
exists almost surely. Moreover, we have
\begin{equation}
\label{eq:f-ito} F(X_{t})=F(X_{0})+\int_{0}^{t}F_x(X_{s})
\ud X_{s}+\frac{1}{2} \int_{0}^{t}
F_{xx}(X_{s}) \ud[X,X]_{s}.
\end{equation}
\end{thm}
%


The rest of the section contains the essential elements of Gaussian
analysis and Malliavin calculus that are used in this paper. See, for
instance, \xch{Refs.}{the references} \cite{n-p, nualart} for further details. In
what follows, we assume that all the random objects are defined on a
complete probability space $(\varOmega,\mathcal{F},\P)$.

\subsection{Isonormal Gaussian processes derived from covariance
functions}\label{isonormal-1}
\vspace*{4pt}
Let $X=\{ X_t \}_{t\in[0,T]}$ be a centered continuous Gaussian
process on the interval $[0,T]$ with $X_0=0$ and continuous covariance
function $R_X(s,t)$. We assume that $\mathcal{F}$ is generated by $X$.
Denote by $\mathcal{E}$ the set of real-valued step functions on
\xch{$[0,T]$,}{$[0,T],$} and let $\HH$ be the Hilbert space defined as the closure of
$\mathcal{E}$ with respect to the scalar product
\[
\langle \mathbf{1}_{[0,t]},\mathbf{1}_{[0,s]}
\rangle_{\HH}=R_{X}(t,s), \quad s,t \in[0,T].
\]
For example, when $X$ is a Brownian motion, $\HH$ reduces to the
Hilbert space\break$L^2([0,T],\ud t)$. However, in general, $\HH$ is
not a
space of functions, for example, when $X$ is a fractional Brownian
motion with Hurst parameter $H \in(\frac{1}{2},1)$ (see \cite{p-t}).
The mapping $\mathbf{1}_{[0,t]}\longmapsto X_t$ can be extended to a
linear isometry between $\HH$ and the Gaussian space $\mathcal{H}_{1}$
spanned by a Gaussian process $X$. We denote this isometry by $\varphi
\longmapsto X(\varphi)$, and $\{ X(\varphi); \, \varphi\in\HH\}$ is
an isonormal Gaussian process in the sense of \cite[Definition
1.1.1]{nualart}, that is, it is a Gaussian family with covariance
function
\begin{equation*}
\begin{split} \E \bigl( X(\varphi_1) X(
\varphi_2) \bigr) &= \langle\varphi_1,\varphi
_2 \rangle_\HH
\\[3pt]
& = \int_{[0,T]^2} \varphi_1(s)
\varphi_2(t) \ud R_X(s,t), \quad \varphi_1,
\varphi_2 \in\mathcal{E}, \end{split} %
\end{equation*}
where $\ud R_X(s,t) = R_X(\ud s,\ud t)$ stands for the measure induced
by the covariance function $R_X$ on $[0,T]^2$. Let $\mathcal{S}$ be the
space of smooth and cylindrical random variables of the form\vspace{2pt}
\begin{equation}
F=f\bigl(X(\varphi_{1}),\ldots,X(\varphi_{n})\bigr),
\label{g1}
\end{equation}
where $f\in C_{b}^{\infty}(\mathbb{R}^{n})$ ($f$ and all its partial
derivatives are bounded). For a~random variable $F$ of the form \
(\ref{g1}), we define its Malliavin derivative as the $\HH$-valued
random variable
\begin{equation*}
DF=\sum_{i=1}^{n}\frac{\partial f}{\partial x_{i}}
\bigl(X(\varphi _{1}),\ldots,X(\varphi_{n})\bigr)
\varphi_{i}\text{. }
\end{equation*}

By iteration, the $m$th derivative $D^{m}F\in L^{2}(\varOmega
;\HH^{\otimes m})$ is defined for every $m\geq2$. For $m\geq1$,
${\mathbb{D}}^{m,2}$ denotes the closure of $\mathcal{S}$ with respect
to the norm $\Vert\cdot\Vert_{m,2}$, defined by the relation\vspace*{2pt}
\begin{equation*}
\Vert F\Vert_{m,2}^{2}\;=\; \E \bigl[ |F|^{2}
\bigr] +\sum_{i=1}^{m}\E \bigl(\Vert
D^{i}F\Vert_{\HH^{\otimes i}}^{2} \bigr) .
\end{equation*}
Let $\delta$ be the adjoint of the operator $D$, also called the
\textsl{divergence operator}. A~random element $u\in L^{2}(\varOmega,\HH)
$ belongs to the domain of $\delta$, denoted $\mathrm{Dom}(\delta)$,
if and only if it satisfies
\begin{equation*}
\big|\E\langle DF,u\rangle_{\HH}\big|\leq c_{u}\,\Vert F
\Vert_{L^{2}}
\end{equation*}
for any $F\in\mathbb{D}^{1,2}$, where $c_{u}$ is a constant depending
only on $u$. If $u\in\mathrm{Dom}(\delta)$, then the random variable
$\delta(u)$ is defined by the duality relationship
\begin{equation}
\E\bigl(F\delta(u)\bigr)= \ E\langle DF,u\rangle_{\HH}, \label{ipp}
\end{equation}
which holds for every $F\in{\mathbb{D}}^{1,2}$. The divergence
operator $\delta$ is also called the
Skorokhod integral because when the Gaussian process $X$ is a Brownian
motion, it coincides with the anticipating stochastic integral
introduced by Skorokhod~\cite{nualart}. We denote $\delta(u)=\int_0^T
u_t \delta X_t$.

For every $q\geq1$, the symbol $\mathcal{H}_{q}$ stands for the $q$th
{\it Wiener chaos} of~$X$, defined as the closed linear subspace of
$L^2(\varOmega)$ generated by the family $\{H_{q}(X(h)) : h\in~\HH,\llVert  h\rrVert  _{ \HH}=1\}$,
where $H_{q}$ is the $q$th Hermite
polynomial defined as
\begin{equation}
\label{hq} H_q(x) = (-1)^q e^{\frac{x^2}{2}}
\frac{d^q}{dx^q} \bigl( e^{-\frac{x^2}{2}} \bigr).
\end{equation}
We write by convention $\mathcal{H}_{0} = \mathbb{R}$. For any $q\geq
1$, the mapping $I^{X}_{q}(h^{\otimes q})=H_{q}(X(h))$ can be extended
to a linear isometry between the symmetric tensor product $ \HH^{\odot
q}$ (equipped with the modified norm $\sqrt{q!}\llVert  \cdot\rrVert
_{ \HH^{\otimes q}}$) and the $q$th Wiener chaos $\mathcal{H}_{q}$. For
$q=0$, we write by convention $I^{X}_{0}(c)=c$, $c\in\mathbb{R}$. For
any $h \in \HH^{\odot q}$, the random variable $I^X_q(h)$ is called a
multiple Wiener--It\^o integral of order $q$. A crucial fact is that if
$\HH= L^2(A,\mathcal{A},\nu)$, where $\nu$ is a $\sigma$-finite and
nonatomic measure on the measurable space $(A,\mathcal{A})$, then $
\HH^{\odot q} = L^2_s(\nu^q)$, where $L^2_s(\nu^q)$ stands for the
subspace of $L^2(\nu^q)$ composed of the symmetric functions. Moreover,
for every $h \in \HH^{\odot q}=L^2_s(\nu^q)$, the random variable
$I^X_q(h)$ coincides with the $q$-fold multiple Wiener--It\^o integral
of $h$ with respect to the centered Gaussian measure (with control
$\nu$) generated by $X$ (see \cite{nualart}). We will also use the
following central limit theorem for sequences living in a fixed Wiener
chaos (see \cite{nualart-peccati, n-o}).


\begin{thm}\label{t.2.2}
Let $\{F_{n}\}_{n\geq1}$ be a sequence of random variables in the
$q$th Wiener chaos, $q\geq2$, such that $\lim_{n\rightarrow\infty}
\E(F_{n}^{2})=\sigma^{2}$. Then, as $n \to\infty$, the following
asymptotic statements are equivalent:
\begin{itemize}
\item[\rm(i)] $F_n$ converges in law to $\mathscr{N}(0,\sigma^2)$.
\item[\rm(ii)] $ \|DF_n\|^2_{\HH} $ converges in $L^2$ to $q \sigma
^{2}$.
\end{itemize}
\end{thm}

To obtain Berry--Esseen-type estimate, we shall use the following
result from \cite[Corollary 5.2.10]{n-p}.
\begin{thm}\label{B-E-bound}
Let $\{F_n\}_{n \ge1}$ be a sequence of elements in the second Wiener
chaos such that $\E(F_n^2) \to\sigma^2$ and $\operatorname{Var} \Vert D F_n
\Vert^2_{\HH} \to0$ as $n \to\infty$. Then, $F_n
\stackrel{\text{law}}{\rightarrow} Z \sim\mathscr{N}(0,\sigma^2)$, and
\begin{equation*}
\begin{split} \sup_{x\in\R} \big\vert\P(F_n
< x) - \P(Z<x) \big\vert\leq\frac
{2}{\E(F_n^2)} \sqrt{ \operatorname{Var}\Vert D
F_n\Vert^2_{\HH}} +\frac{2
\vert\E(F_n^2) - \sigma^2 \vert}{\max\{\E(F_n^2), \sigma^2 \} }.
\end{split} %
\end{equation*}
\end{thm}

\subsection{Isonormal Gaussian process associated with two Gaussian
processes}\label{isonormal-2}
In this subsection, we briefly describe how two Gaussian processes can
be embedded into an isonormal Gaussian process. Let $X_1$ and $X_2$ be
two independent centered continuous Gaussian processes with
$X_1(0)=X_2(0)=0$ and continuous covariance functions $R_{X_1}$ and
$R_{X_2}$, respectively. Assume that $\HH_1$ and $\HH_2$ denote the
associated Hilbert spaces as explained in \xch{Section}{Subsection}~\ref{isonormal-1}.
The appropriate set $\tilde{\mathcal{E}}$ of elementary functions is
the set of the functions that can be written as $\varphi(t,i)=
\delta_{1i}\varphi_1(t) + \delta_{2i} \varphi_2(t)$ for $(t,i) \in
[0,T] \times\{1,2\}$, where $\varphi_1, \varphi_2 \in\mathcal{E}$,
and $\delta_{ij}$ is the Kronecker's delta. On the set
$\tilde{\mathcal{E}}$, we define the inner product\vspace{0pt}
\begin{equation}
\label{mixed-inner-1} %
\begin{split} \langle\varphi, \psi\rangle_{\tilde{\HH}}:
&= \bigl\langle\varphi(\cdot, 1), \psi(\cdot,1) \bigr\rangle_{\HH_1} +
\bigl\langle\varphi(\cdot, 2), \psi (\cdot,2) \bigr\rangle_{\HH_2}
\\
& = \int_{[0,T]^2} \varphi(s,1) \psi(t,1) \ud
R_{X_1}(s,t) + \int_{[0,T]^2} \varphi(s,2) \psi(t,2)
\ud R_{X_2}(s,t), \end{split} %
\end{equation}
where $\ud R_{X_i}(s,t) = R_{X_i}(\ud s,\ud t),\  i=1,2$.

Let $\HH$ denote the Hilbert space that is the completion of
$\tilde{\mathcal{E}}$ with respect to the inner product
\eqref{mixed-inner-1}. Notice that $\HH\cong\HH_1 \oplus\HH_2$,
where $\HH_1 \oplus\HH_2$ is the direct sum of the Hilbert spaces
$\HH_1$ and $\HH_2$, that is, it is a Hilbert space consisting of
elements of the form of ordered pairs $(h_1,h_2) \in\HH_1 \times
\HH_2$ equipped with the inner product
$\langle(h_1,h_2), (g_1,g_2)\rangle_{\HH_1 \oplus\HH_2} := \langle
h_1,g_1\rangle_{\HH_1} + \langle h_2,g_2\rangle_{\HH_2}$.

Now, for any $\varphi\in\tilde{\mathcal{E}}$, we define $X(\varphi):=
X_1(\varphi(\cdot,1)) + X_2(\varphi(\cdot,2))$. Using the independence
of $X_1$ and $X_2$, we infer that $\E ( X(\varphi) X(\psi)
 )= \langle\varphi_1, \psi\rangle_\HH$ for all $\varphi,\psi\in
\tilde{\mathcal{E}}$. Hence, the mapping $X$ can be extended to an
isometry on $\HH$, and therefore $\{X(h), \, h \in\HH\}$ defines an
isonormal Gaussian process associated to the Gaussian processes $X_1$
and $X_2$.

\subsection{Malliavin calculus with respect to (mixed
Brownian) fractional Brownian\\ motion}

The fractional Brownian motion $B^H= \{B_{t}^{H}\}_{t\in \mathbb{R}}$
with Hurst parameter $H\in(0,1)$ is a zero-mean Gaussian process with
covariance function
\begin{equation}
\E\bigl(B_{t}^{H}B_{s}^{H}
\bigr)=R_{H}(s,t)=\frac{1}{2} \bigl(|t|^{2H}+|s|^{2H}-|t-s|^{2H}
\bigr) \,. \label{cov}
\end{equation}
Let $\HH$ denote the Hilbert space associated to the covariance
function $R_H$; see \xch{Section}{Subsection}~\ref{isonormal-1}. It is well known that
for $H=\frac{1}2$, we have $\HH= L^2([0,T])$, whereas for $H>\frac{1}2$,
we have $L^2 ([0,T]) \subset L^{\frac{1}H} ([0,T]) \subset\vert\HH
\vert\subset\HH$, where $\vert\HH\vert$ is defined as the linear
space of measurable functions $\varphi$ on $[0,T]$ such that
\begin{equation*}
\Vert\varphi\Vert^{2}_{\vert\HH\vert}:= \alpha_H \int
_0^T\int_0^T
\big\vert\varphi(s) \big\vert\big\vert\varphi(t) \big\vert\vert t-s \vert^{2H-2} \ud s \ud
t < \infty,
\end{equation*}
where $\alpha_H=H(2H-1)$.

\begin{prop}[\cite{nualart}, Chapter 5] \label{isometries}
Let $\HH$ denote the Hilbert space associated to the covariance
function $R_H$ for $H \in(0,1)$. If $H=\frac{1}{2}$, that is, $B^H$ is
a~Brownian motion, then for any $\varphi, \psi\in\HH= L^2([0,T], \ud
t)$, the inner product of $\HH$ is given by the well-known It\^o
isometry
\begin{equation*}
\E \bigl(B^{\frac{1}{2}}(\varphi) B^{\frac{1}{2}}(\psi) \bigr) = \langle
\varphi, \psi\rangle_{\HH}= \int_0^T
\varphi(t) \psi(t) \ud t.
\end{equation*}
If $H>\frac{1}{2}$, then for any $\varphi, \psi\in\vert\HH\vert$,
we have
\begin{equation}
\E \bigl(B^H(\varphi) B^H(\psi) \bigr) = \langle\varphi,
\psi\rangle _{\HH}= \alpha_H\int_0^T
\int_0^T \varphi(s) \psi(t) |t-s|^{2H-2}
\ud s \ud t.
\end{equation}
\end{prop}

The following proposition establishes the link between pathwise
integral and Skorokhod integral in Malliavin calculus associated to
fractional Brownian motion and will play an important role in our
analysis.

\begin{prop}[\cite{nualart}]\label{prop:div-pathwise}
Let $u=\{u_t\}_{t \in[0,T]}$ be a stochastic process in the space\break
$\mathbb{D}^{1,2}(\vert\HH\vert)$ such that almost surely
\begin{equation*}
\int_0^T\int_0^T
\vert D_s u_t \vert |t-s|^{2H-2}\ud s \ud t<
\infty.
\end{equation*}
Then $u$ is pathwise integrable, and we have
\begin{equation*}
\int_0^T u_t \ud
B^H_t = \int_0^T
u_t \delta B^H_t + \alpha_H
\int_0^T\int_0^T
D_s u_t |t-s|^{2H-2}\ud s \ud t.
\end{equation*}
\end{prop}

For further use, we also need the following ancillary facts related to
the isonormal Gaussian process derived from the covariance function of
the mixed Brownian--fractional Brownian motion. Assume that $X=W+B^H$
stands for a mixed Brownian--fractional Brownian motion with
$H>\frac{1}{2}$. We denote by $\HH$ the Hilbert space associated to the
covariance function of the process $X$ with inner product $\langle
\cdot, \cdot\rangle_{\HH}$. Then a direct application of relation
(\ref{mixed-inner-1}) and Proposition~\ref{isometries} yields the
following facts. We recall that in what follows the notations $I^X_1$
and $I^X_2$ stand for multiple Wiener integrals of orders $1$ and $2$
with respect to isonormal Gaussian process $X$; see \xch{Section}{Subsection}~\ref{isonormal-1}.

\begin{lemma}
\label{lma:inner_product} For any $\varphi_1, \varphi_2, \psi_1, \psi_2
\in L^2([0,T])$, we have
\begin{equation*}
\begin{split} \E \bigl( I_1^X(\varphi)
I_1^X(\psi) \bigr) &= \langle\varphi,\psi
\rangle_{\HH}
\\
&= \int_0^T \varphi(t) \psi(t) \ud t +
\alpha_H \int_0^T\int
_0^T \varphi(s)\psi(t)|t-s|^{2H-2}\ud s
\ud t. \end{split} %
\end{equation*}
Moreover,
\begin{eqnarray*}
&&\E \bigl( I_2^X(\varphi_1 \otimes\varphi_2) I_2^X(\psi_1 \otimes\psi _2) \bigr) \\
&& \quad = 2 \langle\varphi_1 \otimes\varphi_2,\psi_1 \otimes\psi_2\rangle_{\HH^{\otimes2}}\\
&& \quad = \int_{[0,T]^2}\varphi_1(s_1) \psi_1(s_1)\varphi_2(s_2) \psi_2(s_2) \ud s_1 \ud s_2\\
&& \qquad +\, \alpha_H \int_{[0,T]^3} \varphi_1(s_1)\psi_1(s_1) \varphi_2(s_2)\psi_2(t_2) |t_2-s_2|^{2H-2}\ud s_1 \ud s_2 \ud t_2\\
&& \qquad +\, \alpha_H \int_{[0,T]^3} \varphi_1(s_1)\psi_1(t_1) \varphi_2(s_1)\psi_2(s_1) |t_1-s_1|^{2H-2}\ud s_1 \ud t_1 \ud s_1\\
&& \qquad +\, \alpha^2_H \int_{[0,T]^4}\varphi_1(s_1) \psi_1(t_1)\varphi_2(s_2) \psi_2(t_2)\\
&& \qquad  \, \times\vert t_1 - s_1\vert^{2H-2} \vert t_2-s_2 \vert^{2H-2}\ud s_1 \ud t_1 \ud s_2 \ud t_2.
\end{eqnarray*}
\end{lemma}

\section{Main results}\label{beef}

Throughout this section, we assume that $X=W+B^H$ is a mixed
Brownian--fractional Brownian motion with $H>\frac{1}{2}$, unless
otherwise stated. We denote by $\HH$ the Hilbert space associated to
process $X$ with inner product $\langle\cdot, \cdot\rangle_{\HH}$.

\subsection{Central limit theorem}

We start with the following fact, which is one of our key ingredients.

\begin{lemma}[\cite{azm-val}]\label{l:mixed-f}
Let $\E\xi^2 < \infty$. Then the randomized periodogram of the mixed
Brownian--fractional Brownian motion $X$ given by
\textup{(\ref{eq:r-p-m-b-fb})} satisfies
\begin{equation}
\label{rp-path} \E_{\xi} I_{T}(X;L \, \xi)=[X,X]_{T}+2
\int_{0}^{T} \int_{0}^{t}
\varphi _{\xi}\bigl(L(t-s)\bigr) \ud X_{s} \ud
X_{t},
\end{equation}
where $\varphi_{\xi}$ is the characteristic function of $\xi$, and the
iterated stochastic integral in the right-hand side is understood
pathwise, that is, as the limit of the Riemann--Stieltjes sums; see
\xch{Section}{Subsection}~\ref{path-follmer}.
\end{lemma}
Our first aim is to transform the pathwise integral in
(\ref{rp-path}) into the Skorokhod integral. This is the topic of the
next lemma.
%
\begin{lemma}\label{cor:relation}
Let $u_t = \int_{0}^{t}\varphi_\xi(L(t-s)) \ud X_{s}$, where
$\varphi_\xi$ denotes the characteristic function of a symmetric random
variable $\xi$. Then $u\in\mathrm{Dom}(\delta)$, and
\[
\int_0^T u_t \ud X_{t}
= \int_0^T u_t \delta
X_t + \alpha_H \int_0^T
\int_0^T D^{(B^H)}_s
u_t |t-s|^{2H-2}\ud s\ud t,
\]
where the stochastic integral in the right-hand side is the Skorokhod
integral with respect to mixed Brownian--fractional Brownian motion
$X$, and $D^{(B^H)}$ denotes the Malliavin derivative operator with
respect to the fractional Brownian motion $B^H$.
\end{lemma}
\begin{proof}
First, note that
\[
u_t= u^W_t + u^{B^H}_t=
\int_{0}^{t}\varphi_\xi\bigl(L(t-s)
\bigr) \ud W_{s} + \int_{0}^{t}
\varphi_\xi\bigl(L(t-s)\bigr) \ud B^H_{s}.
\]
Moreover, $\E(
\int_{0}^{T} u_t^2 \ud t ) < \infty$, so that $u_t \in
\mathbb{D}^{1,2}$ for almost all $t \in[0,T]$ and $\E(\int_{[0,T]^2}
(D_s u_t)^2 \ud s \ud t) < \infty$. Hence, $u \in\mathrm{Dom}(\delta)$
by \cite[Proposition 1.3.1]{nualart}. On the other hand,
\begin{equation*}
\begin{split}
\int_0^T u_t \ud X_t &= \int_0^T u_t \ud W_t + \int_0^T u_t \ud B^H_t\\
& = \int_0^T u^W_t\ud W_t + \int_0^T u^{B^H}_t \ud W_t + \int_0^T u^{W}_t \ud B^H_t + \int_0^T u^{B^H}_t \ud B^H_t\\
&= \int_0^T u^W_t\delta W_t + \int_0^T u^{B^H}_t \delta W_t + \int_0^T u^W_t \delta B^H_t + \int_0^T u^{B^H}_t \delta B^H_t\\
&\quad  + \alpha_H \int_0^T\int_0^T D^{(B^H)}_s u^{B^H}_t |t-s|^{2H-2}\ud s\ud t\\
&= \int_{0}^{T} u_t \delta W_t + \int_{0}^{T} u_t\delta B^H_t + \alpha _H \int_0^T\int_0^T D^{(B^H)}_s u_t |t-s|^{2H-2}\ud s\ud t,
\end{split} %
\end{equation*}
where we have used the independence of $W$ and $B^H$, Proposition \ref
{prop:div-pathwise}, and the fact that for adapted integrands, the
Skorokhod integral coincides with the It\^o integral. To finish the
proof, we use the very definition of Skorokhod integral and relation
(\ref{ipp}) to obtain that $\int_{0}^{T} u_t \delta W_t + \int_{0}^{T}
u_t \delta B^H_t = \int_{0}^{T} u_t \delta X_t$.
\end{proof}

We will also pose the following assumption for characteristic function
$\varphi_\xi$ of a symmetric random variable $\xi$.
\begin{Assumption}\label{assumption:main}
The characteristic function $\varphi_\xi$ satisfies
\[
\int_0^\infty\big|\varphi_\xi(x)\big|\ud x <
\infty.
\]
\end{Assumption}

\begin{remark}{\rm

Note that Assumption \ref{assumption:main} is satisfied for many
distributions. Especially, if the characteristic function $\varphi_\xi$
is positive and the density function $g_\xi(x)$ is differentiable, then
we get by applying Fubini's theorem and integration by part that
\[
\int_0^\infty\varphi_\xi(x)\ud x = 2
\int_0^\infty\int_0^\infty
cos(yx)g_\xi(y)\ud y\ud x = \pi g_\xi(0) < \infty.
\]
}
\end{remark}

We continue with the following technical lemma, which in fact provides
a~correct normalization for our central limit theorems.

\begin{lemma}\label{normalization-factor}
Consider the symmetric two-variable function $\psi_L(s,t):=
\varphi_\xi(L \vert t-s \vert)$ on $[0,T]\times[0,T]$. Then $\psi_L
\in\HH^{\otimes2}$, and moreover, as $L \to\infty$, we have
\begin{equation}
\label{eq:variance} \lim_{L\rightarrow\infty} L\Vert\psi_L
\Vert^2_{\HH^{\otimes2}}= \sigma^2_{T} < \infty,
\end{equation}
where $\sigma^2_T:= 2\, T \int_{0}^{\infty} \varphi^2_{\xi}(x) \ud x$
is independent of the Hurst parameter $H$.
\end{lemma}

\begin{remark}{ \rm We point it out that the variance $\sigma^2_T$ in
Lemma \ref{normalization-factor} is finite. This is a simple
consequence of Assumption \ref{assumption:main} and the fact that the
characteristic function $\varphi_\xi$ is bounded by one over the real
line.
}
\end{remark}

\begin{proof}
Throughout the proof, $C$ denotes unimportant constant depending on~$T$
and~$H$, which may vary from line to line. First, note that clearly
$\psi_L \in\HH^{\otimes2}$ since $\psi_L$ is a bounded function. In
order to prove \eqref{eq:variance}, we show that, as $L\to\infty$,
\[
\Vert\psi_L \Vert^2_{\HH^{\otimes2}} \sim
\frac{1}{L}.
\]
Next, by applying Lemma \ref{lma:inner_product} we obtain $\Vert\psi_L
\Vert^2_{\HH^{\otimes2}}=A_1 +A_2+A_3$, where
\begin{align}
A_1 &:= \int_{[0,T]^2}\varphi_\xi^2\big(L|t-s|\xch{\big)}{))}
\ud t\ud s,
\\
A_2 &:= \alpha_H\int_{[0,T]^3}
\varphi_\xi\big(L|t-u|\big)\varphi_\xi \big(L|s-u|\big)|t-s|^{2H-2}
\ud t\ud s\ud u,
\\
A_3 & : = \alpha_H^2\int
_{[0,T]^4}\varphi_\xi\big(L|t-u|\big)\varphi_\xi
\big(L|s-v|\big)|t-s|^{2H-2}|v-u|^{2H-2}\ud u\ud v\ud t\ud s.
\end{align}
First, we show that $A_1 \sim\frac{1}{L}$. By change of variables
$y=\frac{L}{T}s$ and $x=\frac{L}{T}t$ we obtain
\[
A_1 = \frac{T^2}{L^2}\int_0^L
\int_0^L \varphi^2_\xi\big(T|x-y|\big)
\ud x\ud y.
\]
Now, by applying L'H\^opital's rule and some elementary computations we
obtain that
\begin{equation*}
\begin{split} \lim_{L\to\infty} L^{-1} \int
_0^L \int_0^L
\varphi^2_\xi\big(T|x-y|\big)\ud x\ud y &= \lim
_{L\to\infty} 2\int_0^L
\varphi^2_\xi\bigl(T(L-x)\bigr)\ud x
\\
&= \frac{2}{T}\int_0^\infty
\varphi^2_\xi(y)\ud y, \end{split} %
\end{equation*}
which is finite by Assumption \ref{assumption:main}. Consequently, we
get
\[
\lim_{L\to\infty} L A_1 = 2T \int_0^\infty
\varphi^2_\xi(y)\ud y,
\]
or, in other words, $A_1 \sim L^{-1}$. To complete the proof, it is
shown in Appendix~B that $\lim_{L\to\infty} L(A_2 + A_3)=0$.
\end{proof}
%

We also apply the following proposition. The proof is rather technical
and is postponed to Appendix A.
\begin{prop}
\label{prop:CLT} Consider the symmetric two-variable function
$\psi_L(s,t):=\break \varphi_\xi(L \vert t-s \vert)$ on $[0,T]\times[0,T]$.
Denote
\[
\tilde{\psi}_L(t,s) = \frac{\psi_L(s,t)}{\sqrt{2} \Vert\psi_L \Vert
_{\HH^{\otimes2}}}.
\]
Then, for any $H\in (\frac{1}{2},1 )$, as $L \to\infty$, we
have
\[
I_2^X (\tilde{\psi}_L ) \stackrel{
\text{law}} {\longrightarrow} \mathscr{N}(0,1).
\]
\end{prop}

Our main theorem is the following.
\begin{thm}\label{main-thm}
Assume that the characteristic function $\varphi_\xi$ of a symmetric
random variable $\xi$ satisfies Assumption~\textup{\ref{assumption:main}} and
let $\sigma_{T}^2$ be given by \eqref{eq:variance}. Then, as $L \to
\infty$, we have the following asymptotic statements:
\begin{enumerate}
\item[\rm1.] if $H\in (\frac{3}{4},1 )$, then
\begin{equation*}
\sqrt{L} \bigl( \E_{\xi} I_{T}(X;L \, \xi)-[X,X]_{T}
\bigr) \stackrel{\text {law}} {\longrightarrow} \mathscr{N}\bigl(0,
\sigma_T^2\bigr).
\end{equation*}

\item[\rm2.] if $H=\frac{3}{4}$, then
\begin{equation*}
\sqrt{L} \bigl( \E_{\xi} I_{T}(X;L \, \xi)-[X,X]_{T}
\bigr) \stackrel{\text {law}} {\longrightarrow} \mathscr{N}\bigl(\mu,
\sigma_T^2\bigr),
\end{equation*}
where $\mu= 2\alpha_HT\int_0^\infty\varphi_\xi(x)x^{2H-2}\ud x$.
\item[\rm3.] if $H\in (\frac{1}{2},\frac{3}{4} )$, then
\begin{equation*}
L^{2H-1} \bigl( \E I_{T}(X; L \, \xi)-[X,X]_{T}
\bigr) \stackrel{\P } {\longrightarrow} \mu,
\end{equation*}
where the real number $\mu$ is given in item $2$. Notice that when
$H\in (\frac{1}{2},\frac{3}{4} )$, we have $2H-1 <
\frac{1}{2}$.
\end{enumerate}
\end{thm}

\begin{proof}
First, by applying Lemmas \ref{l:mixed-f} and \ref{cor:relation} we can
write
\[
\E I_{T}(X;L \xi)-[X,X]_{T} = I^X_2(
\psi_L) + \alpha_H\int_0^T
\int_0^T \varphi_\xi\big(L|t-s|\big)|t-s|^{2H-2}
\ud s\ud t.
\]
Consequently, we obtain
\begin{equation*}
\begin{split}
&\sqrt{L} \bigl( \E I_{T}(X;L\xi)-[X,X]_{T} \bigr) \\
&\quad = \sqrt{L}\, I^X_2(\psi_L) + \sqrt{L} \, \alpha_H\int_0^T\int_0^T \varphi_\xi \big(L|t-s|\big)|t-s|^{2H-2}\ud s\ud t\\
&\quad :=A_1 + A_2.
\end{split} %
\end{equation*}
Now, thanks to Proposition \ref{prop:CLT}, for any
$H\in (\frac{1}2,1 )$, we have
\[
A_1 = \sqrt{L}\, \Vert\psi_L \Vert_{\HH^{\otimes2}}
I_2^X (\tilde{\psi }_L) \stackrel{
\text{law}} {\rightarrow} \mathscr{N}\bigl(0,\sigma_T^2
\bigr),
\]
where $\sigma_H^2$ is given by (\ref{eq:variance}). Hence, it remains
to study the term $A_2$. Using change of variables $y=\frac{L}{T}s$ and
$x=\frac{L}{T}t$, we obtain
\begin{equation*}
\begin{split}
&\int_0^T \int_0^T \varphi_\xi\big(L|t-s|\big)|t-s|^{2H-2}\ud s\ud t  \\
& \quad = T^{2H}L^{-2H}\int_0^L\int_0^L\varphi_\xi\big(T|x-y|\big)|x-y|^{2H-2}\ud x\ud y,
\end{split}
\end{equation*}
where by L'H\^opital's rule we obtain
\[
\lim_{L\rightarrow\infty} L^{-1} \int_0^L
\int_0^L \varphi_\xi
\big(T|x-y|\big)|x-y|^{2H-2}\ud x\ud y = 2T^{1-2H}\int
_0^\infty\varphi_\xi (x)x^{2H-2}
\ud x.
\]
Note also that the integral in the right-hand side of the last identity
is finite by Assumption \ref{assumption:main}. Consequently, we obtain
\begin{align}
\label{eq:trace_limit} %
&\lim_{L\rightarrow\infty}L^{2H-1}\alpha_H  \int_0^T \int_0^T \varphi _\xi\big(L|t-s|\big)|t-s|^{2H-2}\ud s\ud t\nonumber\\
&\quad = 2 \alpha_H T\int_0^\infty\varphi_\xi(x)x^{2H-2}\ud x = \mu.
\end{align}
Therefore,
\[
\lim_{L\rightarrow\infty} A_2 = \lim_{L\rightarrow\infty}
L^{\frac{3}2-2H}\mu,
\]
which converges to zero for $H\in (\frac{3}4,1 )$, and item $1$
of the claim is proved. Similarly, for $H=\frac{3}4$, we obtain
\[
\lim_{L\rightarrow\infty} A_2 = \mu,
\]
which proves item $2$ of the claim. Finally, for item $3$, from
(\ref{eq:trace_limit}) we infer that, as $L \to\infty$,
\[
L^{2H-1}\alpha_H\int_0^T
\int_0^T \varphi_\xi\big(L|t-s|\big)|t-s|^{2H-2}
\ud s \ud t \longrightarrow\mu.
\]
Furthermore, for the term $I_2^X(\psi_L)$, we obtain
\[
L^{2H-1}\, I^X_2(\psi_L) =
L^{2H - \frac{3}2} \times\sqrt{L} \, I^X_2(
\psi_L) \stackrel{\P} {\to} 0
\]
as $L \to\infty$. This is because $H<\frac{3}4$ implies $2H-\frac{3}2 < 0$
and moreover $\sqrt{L}\, I^X_2(\psi_L) \stackrel{\text{law}}{\to}
\mathscr{N}(0,1)$ and $L^{2H-\frac{3}2}\rightarrow0$.
\end{proof}

\begin{cor}
When $X=W$ is a standard Brownian motion, that is, if the fractional
Brownian motion part drops, then with similar arguments as in Theorem~\textup{\ref{main-thm}}, we obtain

\begin{equation*}
\sqrt{L} \bigl( \E_{\xi} I_{T}(X;L \, \xi)-[X,X]_{T}
\bigr) \stackrel{\text {law}} {\longrightarrow} \mathscr{N}\bigl(0,
\sigma_T^2\bigr),
\end{equation*}
where $\sigma_T^2=2 T \int_0^\infty\varphi^2_\xi(x) \ud x$, and
$\varphi_\xi$ is the characteristic function of $\xi$.
\end{cor}

\begin{remark}
{\rm

Note that the proof of Theorem \ref{main-thm} reveals that in the case
$H\in (\frac{1}2,\frac{3}4 )$, for any $\epsilon> \frac{3}{2} -
2H$, we have that, as $L \to\infty$,
\begin{equation*}
\sqrt{L} \bigl( \E I_{T}(X; L \, \xi)-[X,X]_{T} \bigr)
\stackrel{\P } {\longrightarrow} \infty,
\end{equation*}
and, moreover,
\begin{equation*}
L^{\frac{1}{2}-\epsilon} \bigl( \E I_{T}(X; L \, \xi)-[X,X]_{T}
\bigr) \stackrel{\P} {\longrightarrow} 0.
\end{equation*}
}
\end{remark}

\subsection{The Berry--Esseen estimates}
As a consequence of the proof of Theorem \ref{main-thm}, we also obtain
the following Berry--Esseen bound for the semimartingale case.
\begin{prop}
\label{prop:Berry-Esseen} Let all the assumptions of Theorem~\textup{\ref{main-thm}}
hold, and let $H\in (\frac{3}4,1 )$. Furthermore,
let $Z\sim\mathscr{N}(0,\sigma_T^2)$, where the variance $\sigma_T^2$
is given by \textup{(\ref{eq:variance})}. Then there exists a constant $C$
\emph{(}independent of $L)$ such that for sufficiently large $L$, we
have
\[
\sup_{x\in\R} \big|\P (\sqrt{L} \bigl(\E_{\xi}
\bigl(I_{T}(X;L \, \xi )-[X,X]_{T} \bigr) < x \bigr) - \P (Z
< x ) \big| \leq C \rho(L),
\]
where
\[
\rho(L)= \max \Biggl\{ L^{\frac{3}2-2H}, \int_L^\infty
\varphi^2_\xi (Tz)\ud z \Biggr\}.
\]
\end{prop}

\begin{proof}
By proof of Theorem \ref{main-thm} we have
\begin{equation*}
\begin{split}
&\sqrt{L} \bigl(\E I_{T}(X;L\xi)-[X,X]_{T} \bigr) \\[3pt]
&\quad = \sqrt{L} \, I^X_2(\psi_L) +
\sqrt{L}\, \alpha_H \int_0^T\int_0^T \varphi_\xi\big(L|t-s|\big)|t-s|^{2H-2}\ud s\ud t\\[3pt]
&\quad =: A_1 + A_2,
\end{split} %
\end{equation*}
where
\[
A_1 = \sqrt{2 L} \Vert\psi_L\Vert_{\HH^{\otimes2}} \,
I_2^X(\tilde {\psi}_L).
\]
Now, we know that the deterministic term $A_2$ converges to zero with
rate $L^{\frac{3}2 - 2H}$ and the term $A_1
\stackrel{\text{law}}{\to}\mathscr{N}(0,\sigma_T^2)$. Hence, in order
to complete the proof, it is sufficient to show that
\[
\sup_{x\in\R} \big| \P(A_1 < x) - \P(Z<x) \big| \leq C
\rho(L).
\]

Now, by using the proof of Proposition \ref{prop:CLT} in \xch{Appendix}{the Appendix} A
we have
\[
\sqrt{ \operatorname{Var} \Vert D F_L\Vert^2_{\HH}
} \leq L^{-\frac{1}2} \leq L^{\frac{3}2-2H}.
\]
Finally, using the notation of the proof of Lemma
\ref{normalization-factor}, we have
\[
E \bigl(F_n^2\bigr) = L \, \Vert\psi_{L}
\Vert^2_{\HH^{\otimes2}} = L \times (A_1 + A_2
+ A_3),
\]
where $A_2 + A_3 \leq CL^{-2H}$. Consequently,
\[
L \times(A_2+A_3) \leq C L^{1-2H} \leq C
L^{\frac{3}2 -2H}.
\]
To complete the proof, we have
\begin{equation*}
\begin{split} LA_1 &= \frac{T^2}{L}\int
_0^L \int_0^L
\varphi^2_\xi\big(T|x-y|\big)\ud y\ud x =\frac{T^2}{L}\int
_0^L \int_{-x}^{L-x}
\varphi^2_\xi(Tz)\ud z\ud x
\\
&=\frac{T^2}{L}\int_{-L}^L \int
_{-z}^{L-z}\varphi^2_\xi(Tz)
\ud x \ud z = T^2 \int_{-L}^L
\varphi^2_\xi(Tz)\ud z
\\
&=2T^2 \int_0^L
\varphi^2_\xi(Tz)\ud z. \end{split} %
\end{equation*}
This gives us
\begin{equation*}
LA_1 - \sigma_T^2 = 2T^2 \int
_L^\infty\varphi^2_\xi(Tz)
\ud z.
\end{equation*}
Now, the claim follows by an application of Theorem \ref{B-E-bound}.
\end{proof}
\begin{remark}\rm
In many cases of interest, the leading term in $\rho(L)$ is the
polynomial term $ L^{\frac{3}2-2H}$, which reveals that the role of the
particular choice of $\varphi_\xi$ affects only to the constant. In
particular, if $\varphi_\xi$ admits an exponential decay, that is,
$\vert\varphi_\xi(t)\vert\leq C_1e^{-C_2t}$ for some constants
$C_1,C_2>0$, then $\int_L^\infty\varphi_\xi^2(Tz)\ud z \leq
C_3e^{-C_4L} \leq CL^{\frac{3}2 -2H}$ for some constants $C_3,C_4,C>0$.
As examples, this is the case if $\xi$ is a standard normal random
variable with characteristic function $\varphi_\xi(t) =
e^{-\frac{t^2}{2}}$ or if $\xi$ is a standard Cauchy random variable
with characteristic function $\varphi_\xi(t) = e^{- \vert t \vert}$.
\end{remark}

\begin{remark} \rm
Consider the case $X=W$, that is, $X$ is a standard Brownian motion. In
this case, the correction term $A_2$ in the proof of Theorem
\ref{main-thm} disappears, and we have
\[
\E\bigl(F_L^2\bigr) - \sigma_T^2
= 2T^2\int_L^\infty
\varphi^{2}_\xi(Tx)\ud x.
\]
Furthermore, by applying L'H\^opital's rule twice and some elementary
computations it can be shown that
\[
\E \bigl[\Vert D F_L \Vert_{\HH}^2 - \E\Vert
D F_L\Vert_{\HH}^2 \bigr]^2 \leq
\big\vert\varphi_\xi(TL) \big\vert\, L^{-1}.
\]
Consequently, in this case, we obtain the Berry--Esseen bound
\[
\sup_{x\in\R} \big|\P (\sqrt{L} \bigl(\E_{\xi}
\bigl(I_{T}(X;L \, \xi )-[X,X]_{T} \bigr) < x \bigr) - \P (Z
< x ) \big| \leq C \rho(L),
\]
where
\[
\rho(L)= \max \Biggl\{ \sqrt{ \big\vert\varphi_\xi(TL) \big\vert
L^{-1}},\int_L^\infty
\varphi^2_\xi(Tz)\ud z \Biggr\},
\]
which is in fact better in many cases of interest. For example, if
$\varphi_\xi$ admits an exponential decay, then we obtain $\rho(L) \leq
e^{-cL}$ for some constant $c$.
\end{remark}

\section*{Acknowledgments}
Azmoodeh is supported by research project F1R-MTH-PUL-12PAMP from
University of Luxembourg, and Lauri Viitasaari was partially funded by
Emil Aaltonen Foundation. The authors are grateful to Christian Bender
for useful discussions.

\appendix

\section{Appendix section}


\subsection{Proof of Proposition \ref{prop:CLT}}\label{A}
Denote $F_L = I_2^X (\tilde{\psi}_L )$ and note that by the definition
of $\tilde{\psi}_L$ we have $\E( F_L^2) = 1$. Hence, it is sufficient
to prove that, as $L \to\infty$,
\[
\E \bigl[\Vert D F_L\Vert_{\HH}^2 - \E\Vert D
F_L\Vert_{\HH}^2 \bigr]^2
\rightarrow0.
\]
Now, using the definition of the Malliavin derivative, we get
\[
D_s F_L = 2 \, I_1^{X}\bigl(
\tilde{\psi}_L(s,\cdot)\bigr) = \frac{\sqrt{2}}{\Vert
\psi_L \Vert_{\HH^{\otimes2}}} \,
I_1^X\bigl(\varphi_\xi(L|s-\cdot|)\bigr).
\]
For the rest of the proof, $C$ denotes unimportant constants, which may
vary from line to line. Furthermore, we also use the short notation
\[
K(\ud s,\ud t) = \delta_0(t-s) \ud s \ud t + \alpha_H|t-s|^{2H-2}
\ud s\ud t,
\]
where $\delta_0$ denotes the Kronecker delta function, to denote the
measure associated to the Hilbert space $\HH$ generated by the mixed
Brownian--fractional Brownian motion $X$. Furthermore, without loss of
generality, we assume that $\varphi_\xi\geq0$. Indeed, otherwise we
simply approximate the integral by taking absolute values inside the
integral, which is consistent with Assumption \ref{assumption:main}.
Now we have
\begin{equation*}
\begin{split} \Vert D_sF_L
\Vert^2_{\HH} &= \frac{C}{\Vert\psi_L\Vert^{2}_{\HH
^{\otimes2}}}\int_0^T
\int_0^TI_1^X\bigl(
\varphi_\xi\big(L|u-\cdot |\big)\bigr)I_1^X\bigl(
\varphi_\xi\big(L|v-\cdot|\big)\bigr)K(\ud u,\ud v). \end{split} %
\end{equation*}
Next, using the multiplication formula for multiple Wiener integrals,
we see that
\begin{equation*}
\begin{split}
&I_1^X\bigl(\varphi_\xi\big(L|u-\cdot|\big)\bigr)I_1^X\bigl(\varphi_\xi\big(L|v-\cdot|\big)\bigr)\\
&\quad =\bigl\langle\varphi_\xi\big(L|u-\cdot|\big),\varphi_\xi\big(L|v-\cdot|\big)\bigr\rangle_{\HH}+ I_2^{X} \bigl(\varphi_\xi\big(L|u-\cdot|\big)\tilde{\otimes} \varphi_\xi \big(L|v-\cdot|\big) \bigr)\\
&\quad =: J_1(u,v) + J_2(u,v),
\end{split} %
\end{equation*}
%
where the term $J_1$ is deterministic, and $J_2$ has expectation zero.
Hence, we need to show that
\begin{equation}
\label{simple_conv} \E \Biggl[\frac{1}{\Vert\psi_L \Vert^{ 2}_{\HH^{\otimes2}}}\int_0^T
\int_0^T J_2(u,v) K(\ud u,\ud v)
\Biggr]^2\rightarrow0.
\end{equation}
Therefore, by applying Fubini's theorem it suffices to show that, as $L
\to\infty$,
\begin{equation}
\frac{1}{\Vert\psi_L\Vert^{4}_{\HH^{\otimes2}}}\int_{[0,T]^4} \E \bigl[J_2(u_1,v_1)J_2(u_2,v_2)\bigr]
K(\ud u_1,\ud v_1)K(\ud u_2,\ud v_2) \rightarrow0.\label{apu_conv}
\end{equation}
First, using isometry (iii) \cite[p.~9]{nualart} , we get that
\begin{equation*}
\begin{split}
&\E \bigl[J_2 (u_1,v_1)J_2(u_2,v_2) \bigr]\\
&\quad  = 2 \int_{[0,T]^4} \bigl(\varphi _\xi\big(L|u_1-\cdot|\big)\tilde{\otimes}\varphi_\xi\big(L|v_1-\cdot|\big) \bigr)(x_1,y_1)\\
&\qquad  \times \bigl(\varphi_\xi\big(L|u_2-\cdot|\big)\tilde{\otimes}\varphi_\xi \big(L|v_2-\cdot|\big) \bigr) (x_2,y_2)\, K(\ud x_1,\ud x_2)K(\ud y_1,\ud y_2).
\end{split} %
\end{equation*}
By plugging into (\ref{apu_conv}) we obtain that it suffices to have
\begin{align}
&\frac{1}{\Vert\psi_L\Vert^{4}_{\HH^{\otimes2}}}  \int_{[0,T]^8} \bigl(\varphi_\xi\big(L|u_1-\cdot|\big)\tilde{\otimes}\varphi_\xi\big(L|v_1-\cdot|\big) \bigr)(x_1,y_1)\nonumber\\
&\quad \times \bigl(\varphi_\xi\big(L|u_2-\cdot|\big)\tilde{\otimes}\varphi_\xi \big(L|v_2-\cdot|\big) \bigr) (x_2,y_2)\nonumber\\
&\quad \times K(\ud x_1,\ud x_2)K(\ud y_1,\ud y_2)K(\ud u_1,\ud v_1)K(\ud u_2,\ud v_2)\rightarrow0.\label{apu_conv2}
\end{align}
The rest of the proof is based on similar arguments as the proof of
Lemma~\ref{normalization-factor}.\break Indeed, again by the symmetric
property of measures $K(\ud x,\ud y)$ and functions\break
$\varphi_\xi(L|u_1-\cdot|)\tilde{\otimes}\varphi_\xi(L|v_1-\cdot|)$ we
obtain five different terms, denoted by $A_k$,\break $k=1,2,3,4,5$, of the forms
\begin{align*}
A_1 &= \int_{[0,T]^4}\varphi_\xi\big(L|u-x|\big)\varphi_\xi\big(L|u-y|\big)\varphi_\xi \big(L|v-x|\big)\varphi_\xi\big(L|y-v|\big)\ud x\ud y\ud v\ud u,\\
A_2 &= \alpha_H\int_{[0,T]^5}\varphi_\xi\big(L|u-x_1|\big)\varphi_\xi \big(L|u-y|\big)\varphi_\xi\big(L|v-x_2|\big)\varphi_\xi\big(L|y-v|\big)\\
& \quad \times|x_1-x_2|^{2H-2}\ud x_1\ud x_2\ud y\ud v\ud u,\\
A_3 &= \alpha_H^2\int_{[0,T]^6}\varphi_\xi\big(L|u-x_1|\big)\varphi_\xi\big(L|u-y_1|\big)\varphi_\xi\big(L|v-x_2|\big)\varphi_\xi\big(L|y_2-v|\big)\\
& \quad \times |x_1-x_2|^{2H-2}|y_1-y_2|^{2H-2}\ud x_1\ud x_2\ud y_1\ud y_2\ud v\ud u,\\
A_4 &= \alpha_H^3\int_{[0,T]^7}\varphi_\xi\big(L|u_1-x_1|\big)\varphi_\xi\big(L|v_1-y_1|\big)\varphi_\xi\big(L|v-x_2|\big)\varphi_\xi\big(L|y_2-v|\big)\\
& \quad \times|x_1-x_2|^{2H-2}|y_1-y_2|^{2H-2}|u_1-v_1|^{2H-2}\ud x_1\ud x_2\ud y_1\ud y_2\ud v_1\ud u_1 \ud v,\\
A_5 &= \alpha_H^4\int_{[0,T]^8}\varphi_\xi\big(L|u_1-x_1|\big)\varphi_\xi\big(L|v_1-y_1|\big)\varphi_\xi\big(L|u_2-x_2|\big)\\
& \quad \times\varphi_\xi \big(L|y_2-v_2|\big)|x_1-x_2|^{2H-2}|y_1-y_2|^{2H-2}|u_1-v_1|^{2H-2}\\
& \quad \times|u_2-v_2|^{2H-2}\ud x_1\ud x_2\ud y_1\ud y_2\ud v_1\ud u_1 \ud v_2\ud u_2.
\end{align*}
Next, we prove that $A_3 \leq CL^{-3}$. First, by change of variables
we obtain
\begin{align*}
A_3 &= C L^{-4H-2}\int_{[0,L]^6}\varphi_\xi\big(T|u-x_1|\big)\varphi_\xi\big(T|u-y_1|\big)\varphi_\xi\big(T|v-x_2|\big)\varphi_\xi\big(T|y_2-v|\big)\\
& \quad \times |x_1-x_2|^{2H-2}|y_1-y_2|^{2H-2}\ud x_1\ud x_2\ud y_1\ud y_2\ud v\ud u.
\end{align*}
Note that Assumption \ref{assumption:main} implies that $\int_0^L
\varphi_\xi(T|x-y|)\ud x \leq C$, where the constant $C$ does not
depend on $L$ and $y$. Similarly, we have
\[
\int_0^L |x-y|^{2H-2}\ud x \leq
CL^{2H-1},
\]
where again the constant $C$ is independent of $L$ and $y$. Moreover,
we have $\varphi_\xi(T|u-v|)\leq1$ for any $u,v \in\R$. Hence, we can
estimate
\begin{align*}
%
A_3 &\leq C L^{-4H-2}\int_{[0,L]^6} \varphi_\xi\big(T|u-x_1|\big) \,\varphi_\xi \big(T|u-y_1|\big) \, \varphi_\xi\big(T|v-x_2|\big)\\
& \quad \times \varphi_\xi\big(T|y_2-v|\big) \,|x_1-x_2|^{2H-2}|y_1-y_2|^{2H-2}\ud x_1\ud x_2\ud y_1\ud y_2\ud v\ud u\\
&\leq C L^{-4H-2}\int_{[0,L]^6} 1\times\varphi_\xi\big(T|u-y_1|\big) \, \varphi _\xi\big(T|v-x_2|\big)\, \varphi_\xi\big(T|y_2-v|\big)\\
& \quad \times |x_1-x_2|^{2H-2}|y_1-y_2|^{2H-2}\ud x_1\ud x_2\ud y_1\ud y_2\ud v\ud u\\
&= CL^{-4H-2} \int_{[0,L]^4} \varphi_\xi\big(T|v-x_2|\big)\, \varphi_\xi \big(T|y_2-v|\big) \, |y_1-y_2|^{2H-2}\\
& \quad \times \biggl(\int_{[0,L]^2} \varphi_\xi\big(T|u-y_1|\big)\, |x_1-x_2|^{2H-2}\ud u \ud x_1\biggr)\ud x_2\ud y_1\ud y_2\ud v\\
&\leq CL^{-4H-2}\times L^{2H-1}\int_{[0,L]^4}\varphi_\xi\big(T|v-x_2|\big) \, \varphi_\xi\big(T|y_2-v|\big)\\
& \quad \times|y_1-y_2|^{2H-2} \ud x_2\ud y_1\ud y_2\ud v\\
&= CL^{-2H-3}\int_{[0,L]^3} \varphi_\xi\big(T|y_2-v|\big)\, |y_1-y_2|^{2H-2}\\
& \quad \times \Biggl(\int_0^L \varphi_\xi\big(T|v-x_2|\big) \ud x_2 \Biggr)\ud y_1\ud y_2\ud v\\
&\leq CL^{-2H-3}\int_{[0,L]^2} |y_1-y_2|^{2H-2}\Biggl(\int_0^L \varphi _\xi\big(T|y_2-v|\big)\ud v \Biggr) \ud y_1\ud y_2\\
&\leq CL^{-2H-3}\int_{[0,L]^2}|y_1-y_2|^{2H-2}\ud y_1\ud y_2 =CL^{-3}.
%
\end{align*}
To conclude, treating $A_1$, $A_2$, $A_4$, and $A_5$ similarly, we
deduce that
\[
\sum_{k=1}^5 |A_k| \leq
CL^{-3}.
\]
Hence, by applying $\Vert\psi_L\Vert^2_{\HH^{\otimes2}} \sim L^{-1}$
we obtain \eqref{apu_conv2}, which completes the proof.

\subsection{Analysis of the variance}\label{B}\vspace{4pt}
We have
\begin{align}
A_2 &= \alpha_H\int_{[0,T]^3}
\varphi_\xi\big(L|t-u|\big)\varphi_\xi \big(L|s-u|\big)|t-s|^{2H-2}
\ud t\ud s\ud u,
\\[4pt]
A_3 &= \alpha_H^2\int_{[0,T]^4}
\varphi_\xi\big(L|t-u|\big)\varphi_\xi \big(L|s-v|\big)|t-s|^{2H-2}|v-u|^{2H-2}
\ud u\ud v\ud t\ud s,
\end{align}
which, by change of variable, leads to

\begin{align*}
A_2 &= \alpha_HT^{2H+1}L^{-2H-1}\int_{[0,L]^3}\varphi_\xi \big(T|t-u|\big)\varphi_\xi\big(T|s-u|\big)|t-s|^{2H-2}\ud t\ud s\ud u,\\
A_3 &= \alpha_H^2T^{4H}L^{-4H}\\
&\quad \times \int_{[0,L]^4}\varphi_\xi\big(T|t-u|\big)\varphi_\xi\big(T|s-v|\big)|t-s|^{2H-2}|v-u|^{2H-2} \ud u\ud v\ud t\ud s.
\end{align*}
We begin with the term $A_2$. Denote
\[
\tilde{A}_2(L) = \int_{[0,L]^3}\varphi_\xi\big(T|t-u|\big)
\varphi_\xi \big(T|s-u|\big)|t-s|^{2H-2}\ud t\ud s\ud u.
\]
By differentiating we get
\begin{align*}
\frac{\ud\tilde{A}_2}{\ud L}(L)&= 2 \int_{[0,L]^2}\varphi_\xi \big(T|L-u|\big)\varphi_\xi\big(T|u-v|\big)|L-v|^{2H-2}\ud v\ud u\\
&\quad + \int_{[0,L]^2} \varphi_\xi\big(T|L-u|\big)\varphi_\xi\big(T|L-v|\big)|u-v|^{2H-2}\ud u\ud v\\
&=:J_1+J_2.
\end{align*}
First, we analyze the term $J_1$. Similarly to Appendix A, we assume
that $\varphi_\xi\geq0$. Hence, we have
\begin{equation*}
\begin{split}
\frac{1}{2} J_1 &= \int_{[0,L]^2}\varphi_\xi\big(T|L-u|\big)\varphi_\xi\big(T|u-v|\big)|L-v|^{2H-2}\ud v\ud u\\
&= \int_{[0,L]^2}\varphi_\xi(Tu)\varphi_\xi\big(T|u-v|\big)v^{2H-2}\ud v\ud u\\
&=\int_0^L \int_1^L\varphi_\xi(Tu)\varphi_\xi\big(T|u-v|\big)v^{2H-2}\ud v\ud u\\
&\quad + \int_0^L \int_0^1\varphi_\xi(Tu)\varphi_\xi\big(T|u-v|\big)v^{2H-2}\ud v\ud u\\
&\leq\int_0^L \int_1^L\varphi_\xi(Tu)\varphi_\xi\big(T|u-v|\big)\ud v\ud u + \int_0^L \int_0^1\varphi_\xi(Tu)v^{2H-2}\ud v\ud u\\
&\leq C.
\end{split} %
\end{equation*}
For the term $J_2$, we write
\begin{equation*}
\begin{split} J_2 &= \int_{[0,L]^2}
\varphi_\xi\big(T|L-u|\big)\varphi_\xi \big(T|L-v|\big)|u-v|^{2H-2}
\ud u\ud v
\\[2pt]
&= \int_{[0,L]^2} \varphi_\xi(Tu)
\varphi_\xi(Tv)|u-v|^{2H-2}\ud u\ud v
\\[2pt]
&=2 \int_0^L \int_0^t
\varphi_\xi(Tu)\varphi_\xi(Tv) (v-u)^{2H-2}\ud u
\ud v
\\[2pt]
&= 2 \Biggl(\int_0^1\int_0^t
+ \int_1^L \int_0^{t-1}
+ \int_1^L \int_{t-1}^t
\Biggr)\varphi_\xi(Tu)\varphi_\xi(Tv)
(v-u)^{2H-2}\ud u\ud v
\\
&=: J_{2,1}+J_{2,2}+J_{2,3}. \end{split}
\end{equation*}
Now, it is straightforward to show that $J_{2,1}+J_{2,2} \leq C$.
Consequently, as $L \to\infty$, we obtain
\[
A_2 \sim L^{-2H-1} \tilde{A}_2 \sim
L^{-2H}(J_1+J_{2,1}+J_{2,2}+J_{2,3}),
\]
where
\[
L^{-2H}(J_1+J_{2,1}+J_{2,2}) \sim
L^{-2H}.
\]
For the term $J_{2,3}$, we write
\[
J_{2,3}(L)= \int_1^L \int
_{t-1}^t\varphi_\xi(Tu)
\varphi_\xi (Tv) (v-u)^{2H-2}\ud u\ud v,
\]
so that
\begin{equation*}
\begin{split} \frac{\ud J_{2,3}}{\ud L}(L)&= \int_{L-1}^L
\varphi_\xi(TL)\varphi_\xi (Tv) (L-v)^{2H-2}\ud v
\\
&\leq C\varphi_\xi(TL). \end{split} %
\end{equation*}
Hence, by L'H\^opital's rule we have $L^{-2H}J_{2,3} \sim
L^{1-2H}\varphi_\xi(TL)$. On the other hand, we have $\varphi_\xi(TL) =
o (L^{2H-2} )$ since $\varphi$ is integrable by Assumption
\ref{assumption:main}. Hence,\break $L^{-2H}J_{2,3} = o (L^{-1} )$,
which shows that $\lim_{L\to\infty} LA_2 =0$. Consequently, $A_2$ does
not affect the variance. The term $A_3$ is easier and can be treated
with similar elementary computations together with L'H\^opital's rule.
As a consequence, we obtain $A_3 \sim L^{-2H}$, so that $\lim_{L\to
\infty} L A_3 = 0$. Hence, $A_3$ does not affect the variance either,
which justifies \eqref{eq:variance}.

\end{document}